\documentclass[11pt]{article}
\usepackage{amsthm, amsmath, amssymb, amsfonts, url, booktabs, tikz, setspace, fancyhdr, bm}
\usepackage{geometry}
\usepackage{hyperref, enumerate}
\usepackage[shortlabels]{enumitem}
\usepackage[babel]{microtype}
\usepackage[english]{babel}
\usepackage[capitalise]{cleveref}
\usepackage{comment}
\usepackage{bbm}
\usepackage{csquotes}
\usepackage{mathabx}
\usepackage{tikz}
\usepackage{graphicx}
\usepackage{float}
\usepackage{xcolor}
\usepackage{mathtools}
\usetikzlibrary{positioning, arrows.meta, shapes.geometric}

\counterwithin{figure}{section}

\newtheorem{theorem}{Theorem}[section]

\newtheorem{conj}[theorem]{Conjecture}
\newtheorem{lemma}[theorem]{Lemma}

\newtheorem{claim}[theorem]{Claim}

\usetikzlibrary{decorations.pathmorphing}

\theoremstyle{definition}
\newtheorem{defn}[theorem]{Definition}
\newtheorem*{defn-non}{Definition}

\newlist{Case}{enumerate}{2}
\setlist[Case, 1]{%
    label           =   {\bfseries Case \arabic*.},
    labelindent=1em ,labelwidth=1.3cm, labelsep*=1em, leftmargin =!
}
\setlist[Case, 2]{%
    label           =   {\bfseries Subcase \arabic{Casei}.\arabic*.},
    labelindent=-1em ,labelwidth=1.3cm, labelsep*=1em, leftmargin =!
}

\usepackage{todonotes}

\newcommand{\ex}{\mathrm{ex}}

\begin{document}
%%%%%%%%%%%%%%%%%%%%%%%%%%%%%%%%%%%%%%%%%%%%%%%%%%%%%%%
%\title{\bf\Large  Tur\'an number of $\mathcal{K}^r_{\ell+1}$ with bounded matching number}
\title{A hypergraph analogue of Alon-Frankl Theorem}
\author{%
    Caihong Yang\textsuperscript{\textdagger}\textsuperscript{\textdaggerdbl},%
    \quad Jiasheng Zeng\textsuperscript{\S},%
    \quad Xiao-Dong Zhang\textsuperscript{\S}%
} 
\footnotetext[1]{\scriptsize%
\noindent\textsuperscript{\textdagger}School of Mathematics and Statistics, Fuzhou University, Fuzhou 350108, China;%
\textsuperscript{\textdaggerdbl}Extremal Combinatorics and Probability Group (ECOPRO), Institute for Basic Science (IBS), Daejeon 34126, South Korea. Email: chyang.fzu@gmail.com }
\footnotetext[2]{\scriptsize%
\noindent\textsuperscript{\S}School of Mathematical Sciences, Shanghai Jiao Tong University, Shanghai 200240, China. Email: jasonzeng@sjtu.edu.cn  }
\footnotetext[3]{\scriptsize%
\noindent\textsuperscript{\S}School of Mathematical Sciences,  MOE-LSC and SHL-MAC, Shanghai Jiao Tong University, Shanghai 200240, China. Email: xiaodong@sjtu.edu.cn  }

\date{\today}
%%%%%%%%%%%%%%%%%%%%%%%%%%%%%%%%%%%%%%%%%%%%%%%%%%%%

\maketitle
%\footnote{footnote}
%%%%%%%%%%%%%%%%%%%%%%%%%%%%%%%%%%%%%%%%%%%%%%%%%
\begin{abstract}

Recently, Alon and Frankl (JCTB, 2024) determined  the maximum number of edges in $K_{\ell+1}$-free $n$-vertex graphs with bounded matching number. For integers $\ell\ge r \ge 2$, the family $\mathcal{K}_{\ell+1}^{r}$ consists of all $r$-graphs $F$ with at most $\binom{\ell+1}{2}$ edges such that, for some $(\ell+1)$-set $K$, every pair $\{x,y\} \subseteq K$ is covered by an edge in $F$.  In this paper, we study the maximum number of edges in $\mathcal{K}_{\ell+1}^r$-free  $r$-uniform hypergraphs that have the matching number at most $s$, that is, $\mathrm{ex}_r(n, \{\mathcal{K}_{\ell+1}^r, M^r_{s+1}\})$, and  obtain the exact value for sufficiently large $n$, along with the corresponding extremal hypergraph. This result can be viewed as a hypergraph extension of the work of Alon and Frankl. In addition, for  the $3$-uniform Fano plane $\mathbb{F}$, we  determine the exact value of $\mathrm{ex}_3(n, \{\mathbb{F}, M^3_{s+1}\})$, and characterize the corresponding extremal hypergraph. 

\end{abstract}

{\bf{Key words:}}{ Tur\'an problem, hypergraph, matching number, Fano plane}

%%%%%%%%%%%%%%%%%%%%%%%%%%%%%%%%%%%%%%%%%%%%%%%%%
\section{Introduction}\label{SEC:Introduction}
Let $\mathcal{F}$ be a family of graphs. Denote by $\mathrm{ex}(n,\mathcal{F})$ the maximum number of edges in an $n$-vertex graph that does not contain any graph from $\mathcal{F}$ as a subgraph. If $\mathcal{F} = \{F\}$, we simply write $\mathrm{ex}(n,F)$ instead of $\mathrm{ex}(n,\mathcal{F})$. Determining the value of $\mathrm{ex}(n,\mathcal{F})$ is one of the central problems in Extremal Graph Theory. Let $K_{\ell+1}$ and $M_{s+1}$ denote the complete graph on $\ell+1$ vertices and the matching of size $s+1$, respectively. Let $T(n,\ell)$ be the $\ell$-partite graph on $n$ vertices with the maximum number of edges, and write $t(n,\ell)=|T(n,\ell)|$. Two classical results in this area are Tur\'{a}n’s Theorem~\cite{T41}, stating that $\mathrm{ex}(n,K_{\ell+1}) = t(n,\ell)$, and the Erd\H{o}s--Gallai Theorem~\cite{EG59}, which gives $$\mathrm{ex}(n,M_{s+1}) = \max\left\{ s(n-s)+\binom{s}{2}, \binom{2s+1}{2} \right\}.$$  Moreover, the extremal graphs are $K_{2s+1}$ with some isolated vertices or the joining of the complete graph $K_s$ and the complement graph $\overline{K}_{n-s}$ of the complete graph $K_{n-s}$.
 
Building upon these classical results, Alon and Frankl~\cite{AF24} recently determined the maximum number of edges in a graph with both bounded clique number and bounded matching number. Let $G(n,\ell,s)$ denote the complete $\ell$-partite graph on $n$ vertices in which one partite set has size $n-s$, and the remaining $\ell-1$ parts induce a copy of $T(s,\ell-1)$. We write $g(n,\ell,s)$ for the number of edges in $G(n,\ell,s)$.
\begin{theorem}[Alon and Frankl~\cite{AF24}]\label{AF2024}
For $n \ge 2s + 1$ and $\ell \ge 2$,
$$
\mathrm{ex}(n,\{K_{\ell+1}, M_{s+1}\}) = \max \{t(2s+1,\ell),\, g(n,\ell,s)\}.
$$
\end{theorem}
 Theorem \ref{AF2024} was also obtained in ~\cite{FWY24}. Motivated by this result, Gerbner~\cite{G2024} generalized the theorem of Alon and Frankl. 
When $n$ is sufficiently large, he determined, up to an additive constant, the value of $\mathrm{ex}(n,\{F, M_{s+1}\})$ for graphs $F$ with $\chi(F) > 2$. Later, Xue and Kang~\cite{xue2024generalized} further established stability results for generalized Turán problems with bounded matching number. Using these stability results, they obtained the exact values of $\mathrm{ex}(n,K_r,\{F, M_{s+1}\})$ for $F$ being any non-bipartite graph or a path. Independently, Zhu and Chen~\cite{zhu2025extremal}, using a different approach, also determined the exact value of $\mathrm{ex}(n,\{F, M_{s+1}\})$ for $F$ being any non-bipartite graph or some bipartite graphs. For the case $\chi(F)=2$, several related results are also known. Let $C_k$ denote the cycle of length $k$, and write $C_{\ge k} = \{C_k, C_{k+1}, \ldots\}$ for the collection of all cycles on at least $k$ vertices. Zhao and Lu~\cite{zhao_and_Lu2024generalized} determined the exact value of $\mathrm{ex}(n,\{C_{\ge k}, M_{s+1}\})$, and described the corresponding extremal graphs for sufficiently large $n$. Moreover, Luo, Zhao, and Lu~\cite{LZL2025} determined the exact value of $\mathrm{ex}(n,\{K_{\ell,t}, M_{s+1}\})$ for sufficiently large $s$ and $n$, and for every $3 \le \ell \le t$.

The Tur\'{a}n problem can also be defined for hypergraphs. Let $r \ge 2$, and let $\mathcal{F}$ be a family of $r$-uniform hypergraphs ($r$-graphs for short). Given an $r$-graph $G$, we say that $G$ is \emph{$\mathcal{F}$-free} if it contains no member of $\mathcal{F}$ as a subhypergraph. The \emph{Tur\'{a}n number} $\mathrm{ex}_r(n,\mathcal{F})$ is defined as the maximum number of edges in an $\mathcal{F}$-free $n$-vertex $r$-graph (when $r=2$, this coincides with the Tur\'{a}n number for graphs defined above). Determining $\mathrm{ex}_r(n,\mathcal{F})$ for various families $\mathcal{F}$ is a central problem in extremal combinatorics. While much is known about $\mathrm{ex}(n,\mathcal{F})$, the case $r \ge 3$ is significantly harder; very little is known about $\mathrm{ex}_r(n,\mathcal{F})$ in general. For further results before 2011, we refer to an excellent survey~\cite{KE11}.
 
Two natural generalizations of complete graphs to $r$-uniform hypergraphs play an important role in extremal hypergraph theory. For integers $\ell\ge r \ge 2$, the family $\mathcal{K}_{\ell+1}^{r}$ consists of all $r$-graphs $F$ with at most $\binom{\ell+1}{2}$ edges such that, for some $(\ell+1)$-set $K$, every pair $\{x,y\} \subseteq K$ is covered by an edge in $F$. 
In particular, $\mathcal{K}_{\ell+1}^{(2)}$ contains only the ordinary complete graph $K_{\ell+1}$. Another important hypergraph, denoted $H_{\ell+1}^{r}$, is obtained from the complete $2$-graph $K_{\ell+1}$ by enlarging each edge with a set of $r-2$ new vertices. When $r=2$, $H_{\ell+1}^{r}$ coincides with the complete graph $K_{\ell+1}$. Observe that $H_{\ell+1}^{r} \in \mathcal{K}_{\ell+1}^{r}$, so $H_{\ell+1}^{r}$ and any other graph in $\mathcal{K}_{\ell+1}^{r}$ can be regarded as natural generalizations of complete graphs. Mubayi~\cite{M06} determined the exact Tur\'{a}n number of $\mathcal{K}_{\ell+1}^{r}$ and characterized the extremal hypergraph for sufficiently large $n$. Let $\ell \ge r \ge 2$ be integers, and let $V_1 \cup \cdots \cup V_\ell$ be a partition of $[n]$ where each $V_i$ has size either $\lfloor n/\ell \rfloor$ or $\lceil n/\ell \rceil$. For simplicity, we write $n/\ell$ to denote the size of each part, ignoring floors and ceilings. The \emph{generalized Tur\'{a}n graph} $T_r(n,\ell)$ consists of all $r$-subsets of $[n]$ that contain at most one vertex from each $V_i$, and we denote $t_r(n,\ell) = |T_r(n,\ell)| \sim \binom{\ell}{r} \left(\frac{n}{\ell}\right)^r$.  

\begin{theorem}[Mubayi~\cite{M06}]\label{exact turan huaK}
Let $n \ge 1$ and $\ell \ge r \ge 2$ be integers. Then $$\mathrm{ex}_r(n,\mathcal{K}_{\ell+1}^{r}) = t_r(n,\ell),$$ with $T_r(n,\ell)$ being the unique extremal hypergraph for sufficiently large $n$.
\end{theorem}
Mubayi~\cite{M06} also determined the Tur\'{a}n function of $H_{\ell+1}^{r}$ asymptotically. Later, Pikhurko~\cite{pikhurko2013Hl} obtained the exact value of $\mathrm{ex}_r(n,H_{\ell+1}^{r})$ for all sufficiently large $n$. In 2025, Hou, Li, Yang, Zeng, Zhang~\cite{HLYZZ25} have studied two kinds of stability theorems for $\mathcal{K}^r_{\ell}$-saturated hypergraphs. We call an $r$-graph $H$ a Berge-$G$ if there exists a bijection $f : E(G) \to E(H)$ such that $e \subseteq f(e)$ holds for every $e \in E(G)$. Gerbner and Palmer~\cite{gerbner2017extremalSIADM} investigated the Tur\'an number of Berge-$G$ hypergraphs when $G$ is an arbitrary graph. We note here that every Berge-$G$ necessarily has exactly $e(G)$ hyperedges due to the bijective correspondence in the definition. In particular, $H_{\ell+1}^r$ is a special instance of a Berge-$K_{\ell+1}$ hypergraph.
 
For an $r$-graph $\mathcal{H}$, the \emph{matching number} $\nu(\mathcal{H})$ is defined as the maximum number of pairwise disjoint edges in $\mathcal{H}$. Let $M_{s+1}^r$ denote a matching of size $s+1$, where each edge is an $r$-set. The Tur\'{a}n problem for matchings in hypergraphs has attracted considerable attention. In 1965, Erd\H{o}s~\cite{Erdos65} proposed the \emph{Erd\H{o}s Matching Conjecture}, which concerns the Tur\'{a}n number of $M_{s+1}^r$. The conjecture asserts that for sufficiently large $n$, the maximum number of edges in an $M_{s+1}^r$-free $r$-graph is achieved by taking all $r$-sets that intersect a fixed $s$-set. This conjecture was later confirmed for large $n$ by Frankl~\cite{Frankl13}, who proved the following exact result.

\begin{theorem}[Frankl~\cite{Frankl13}]
Let $r, s \ge 1$ and $n \ge (2s+1)r - s$. Then $$\mathrm{ex}_r(n, M_{s+1}^r) =  \binom{n}{r} - \binom{n-s}{r}.$$ Equality holds only for families isomorphic to $\mathcal{A}(n,r,s)$, where $\mathcal{A}(n,r,s)$ is the $r$-graph consisting of all hyperedges that intersect a fixed $s$-set.
\end{theorem}
Frankl, Nie, and Wang~\cite{frankl2026matchingdm} studied a variant of the Erd\H{o}s Matching Problem in the setting of random hypergraphs. Gerbner, Tompkins, and Zhou~\cite{GTZ2025} studied the analogous Tur\'{a}n problems for $r$-uniform hypergraphs with bounded matching number, and obtained exact results for certain hypergraphs. Wang, Wang, and Yang~\cite{WWY2025} showed that an $F_5$-free $3$-graph $H$ with matching number at most $s$ has at most $s \left\lfloor \frac{(n-s)^2}{4} \right\rfloor$ edges for $n \geq 30(s+1)$ and $s \geq 3$, where the generalized triangle $F_5$ denotes the $3$-graph on the vertex set $\{a,b,c,d,e\}$ with edge set $\{abc, abd, cde\}$. More recently, for the $3$-graph $F_{3,2}$ with edge set $\{123,145,245,345\}$, Chen, Liu, Qi and Yang~\cite{CLQY25} determined the exact value of $\ex(n,\{F_{3,2}, M_{s+1}^3\})$ for every integers $s$ and all $n \ge 12s^2$. 

In this paper, we first determine the exact value of $\mathrm{ex}(n, \{\mathcal{K}_{\ell+1}^r, M^r_{s+1}\}),$ and characterize the extremal hypergraph for sufficiently large $n$. This can be regarded as a generalization of the result of Alon and Frankl~\cite{AF24} to $r$-uniform hypergraphs. Before stating our main result, we introduce some basic definitions.

Given an $r$-graph $\mathcal{H}$, the \emph{link} of $v$ in $\mathcal{H}$ is $$L_{\mathcal{H}}(v)=\left\{A\in \binom{V(\mathcal{H})}{r-1}:\{v\}\cup A\in \mathcal{H}\right\}.$$ 
Clearly, $L_{\mathcal{H}}(v)$ can be considered as an $(r-1)$-graph. The \emph{degree} of $v$ is $d_{\mathcal{H}}(v) := |L_{\mathcal{H}}(v)|$. Next,  we provide a construction of the extremal hypergraph of $\{\mathcal{K}^r_{\ell+1}, M^r_{s+1}\}$. 
\begin{defn}\label{def of Gn,l,s,r}
    Let $\ell\geq r\geq 3$, $s\geq 1$ be integers and $n\geq s+\ell$. The $r$-graph $\mathcal{G}=\mathcal{G}(n,\ell,s,r)$ is defined as follows. 
    \begin{itemize}
    \item[(a)] The vertex set $|V(\mathcal{G})|=n$ and $V(\mathcal{G})=V_0\cup V_1\cup \cdots\cup V_{\ell-1}$, where $|V_0|=s$.% and $|V_i| = \frac{n-s}{\ell-1}$ for each $i \in [\ell-1]$.
    \item[(b)] Every edge in $\mathcal{G}$ contains exactly one vertex from $V_0$ and the remaining $r-1$ vertices are distributed across distinct parts among $V_1, \dots, V_{\ell-1}$.
    
        \item[(c)] For each $v\in V_0$, the link of $v$ in $\mathcal{G}$ is $L_{\mathcal{G}}(v)=T_{r-1}(n-s,\ell-1)$.
    \end{itemize}
\end{defn}
One of our main results is as follows:
\begin{theorem}\label{main theorem}
Let $\ell \geq r \geq 3$ and $s \geq 1$ be integers. For $n$ sufficiently large, we have
$$
\mathrm{ex}_r(n,\{\mathcal{K}_{\ell+1}^r,M_{s+1}^r\}) = s \cdot t_{r-1}(n-s, \ell-1).  
$$
 Moreover, $\mathcal{G}(n,\ell,s,r)$ is the unique extremal hypergraph.
\end{theorem}
A particularly important example of a $3$-uniform hypergraph is the Fano plane. Let $\mathbb{F}$ denote the Fano plane, i.e., $\mathbb{F}$ is a $3$-graph on $7$ vertices with edge set
$$
\{123, 345, 561, 174, 275, 376, 246\}.
$$
F\"{u}redi and Simonovits~\cite{FS05}, and independently Keevash and Sudakov~\cite{KS05}, determined the Tur\'{a}n number of the Fano plane for large $n$. In 2019, Bellmann and Reiher~\cite{BR19} provided a complete solution to Tur\'{a}n's hypergraph problem for the Fano plane. Specifically, they proved that for $n \ge 8$, among all $3$-graphs on $n$ vertices not containing the Fano plane, there is a unique extremal hypergraph achieving the maximum number of edges, namely the balanced, complete, bipartite hypergraph. Moreover, for $n = 7$, there exists exactly one other extremal configuration with the same number of edges: the hypergraph obtained from a clique of order $7$ by removing all five edges containing a fixed pair of vertices. Our next result determines $\mathrm{ex}(n, \{\mathbb{F}, M_{s+1}^3\})$ and, for sufficiently large $n$, also characterizes the corresponding extremal hypergraphs. We define the extremal hypergraph of $\{\mathbb{F}, M^3_{s+1}\}$ as in the following.
\begin{defn}
Let $X$ and $Y$ be two disjoint vertex sets with $|X|=s$ and $|Y|=n-s$. Define a $3$-graph $\mathcal{F}(n,s)$ on the vertex set $X \cup Y$, where the edge set is given by
$$
E(\mathcal{F}(n,s)) \coloneqq \{x_1x_2y \cup xy_1y_2 : x,x_1,x_2 \in X, y,y_1,y_2 \in Y\}
$$
\end{defn}
 \begin{theorem}\label{turan fano plane}
 Let $\mathbb{F}$ be the Fano plane. For $n \geq 20s(s+1)$,
 $$
\mathrm{ex}_3(n,\{\mathbb{F},M_{s+1}^3\}) = \binom{s}{2}(n-s) + s \cdot \binom{n-s}{2}. 
 $$
 Moreover, $\mathcal{F}(n,s)$ is the unique extremal hypergraph.
 \end{theorem}
The remainder of the paper is organized as follows. In Section~2, we provide additional definitions and several useful lemmas. In Section~3, we present the proof of Theorem~\ref{main theorem}. In Section~4, we give the proof of Theorem~\ref{turan fano plane} and discuss the case of $H_{\ell+1}^r$.
 
%%%%%%%%%%%%%%%%%%%%%%%%%%%%%%
%%%%%%%%%%%%%%%%%%%%%%%%%%%%%%%%%%%%%%%%%%%%%%%
\section{Preliminaries}\label{SEC:prelim}

Let $r \geq 3$ and let $\mathcal{H} = (V, E)$ be an $r$-graph. The number of edges of $\mathcal{H}$ is denoted as $|\mathcal{H}|$. For a vertex $x \in V$, let $d_{\mathcal{H}}(x)$ denote the number of $r$-edges in $\mathcal{H}$ containing $x$. The maximum degree $\Delta(\mathcal{H})$ of $\mathcal{H}$ is defined as the maximum value among all $d_{\mathcal{H}}(x)$. We call subsets $X, Y \subseteq V(\mathcal{H})$ \emph{strongly independent} and \emph{weakly independent}, respectively, 
if no two vertices in $X$ lie in a common hyperedge, and $Y$ contains no hyperedge of $\mathcal{H}$. For a subset $A \subseteq V(\mathcal{H})$ and $i\in [r]$, we define $d_A(x_1,\cdots,x_i)$ as the number of hyperedges containing the vertices $x_1,\cdots,x_i$ such that the remaining $r-i$ vertices all lie in $A$. More formally,
$$
d_A(x_1,\cdots,x_i) = \big|\{e \in E(\mathcal{H}) : \{x_1,\cdots,x_i\} \subseteq e \text{ and } e \setminus \{x_1,\cdots,x_i\} \subseteq A\}\big|.
$$
Furthermore, for $x \in V(\mathcal{H})$, we denote by $L_A(x)$ the link of $x$ restricted to $A$. For any $\ell \ge r$, let $V_1, V_2, \dots, V_\ell$ be pairwise disjoint vertex classes. We write $T_r(V_1, V_2, \dots, V_\ell)$ for the complete $\ell$-partite $r$-uniform hypergraph, where every hyperedge contains at most one vertex from each $V_i$.  For a positive integer $k$, a mapping $\phi: E \to [k]$ is called a proper edge $k$-coloring of $\mathcal{H}$ if for any two edges $e_1, e_2 \in E$, $\phi(e_1) = \phi(e_2)$ implies $e_1 \cap e_2 = \emptyset$. The edge chromatic number $\chi'(\mathcal{H})$ is defined as the smallest positive integer $k$ for which $\mathcal{H}$ admits a proper edge $k$-coloring. The following lemma can be regarded as a hypergraph version of Vizing's Theorem.
\begin{lemma}\label{Hypergraph Vizing}
Let $r \geq 3$ and let $\mathcal{H}$ be an $r$-graph. Then
$\chi'(\mathcal{H}) \leq (\Delta(\mathcal{H}) - 1)r + 1.$
\end{lemma}
\begin{proof}
    We define the auxiliary graph $L = L(\mathcal{H})$, also known as the \emph{line graph} of $\mathcal{H}$, whose vertex set $V(L) = E(\mathcal{H})$. For any two elements $e_1, e_2 \in E(\mathcal{H})$, $e_1$ and $e_2$ form an edge in $L$ if and only if $e_1 \cap e_2 \neq \emptyset$. Note that $L$ is a $2$-graph (an ordinary graph).  It is easy to see that $\chi'(\mathcal{H}) = \chi(L)$. From the definition of $L$, we observe that every proper vertex $k$-coloring of $L$ corresponds bijectively to a proper edge $k$-coloring of $\mathcal{H}$. Therefore, we have
$$
\chi'(\mathcal{H}) = \chi(L) \leq \Delta(L) + 1.
$$
For any vertex $e \in V(L)$ (which corresponds to an edge in $\mathcal{H}$), its degree $d_L(e)$ equals the number of edges in $E(\mathcal{H})$ that have a non-empty intersection with $e$. Moreover, for any vertex $x \in e$ in the original hypergraph, we have $d_{\mathcal{H}}(x) \leq \Delta(\mathcal{H})$. Therefore,
$$
d_L(e)\leq \sum_{x\in e}(d_{\mathcal{H}}(x)-1)\leq r(\Delta(\mathcal{H})-1)
$$
holds for each $e\in V(L)$. Thus, $\chi'(\mathcal{H})\leq r(\Delta(\mathcal{H})-1)+1$.
\end{proof}

The following lemma demonstrates that in an $r$-graph $\mathcal{H}$ with a given matching number, the number of high-degree vertices cannot be too large.

\begin{lemma}\label{Number of Max Degree is Bounded}
Let $\mathcal{H}$ be an $r$-graph with matching number $\nu(\mathcal{H})=s$. Then the number of vertices with degree exceeding $r(s+1)n^{r-2}$ is at most $s$.
\end{lemma}
\begin{proof}
Denote the vertices with degree exceeding $r(s+1)n^{r-2}$ as $X$. Suppose, on the contrary, that $|X| > s$. Let $V_0 = \{v_1, v_2, \ldots, v_{s+1}\}$ be a subset of $V$. For each $v_i \in V_0$, $i\in [s+1]$, we iteratively select $u_{i_1}\cdots u_{i_{r-1}} \in L_{\mathcal{H}}(v_i)$ such that
\begin{itemize}
    \item $\{u_{i_1},\cdots, u_{i_{r-1}}\} \cap V_0 = \emptyset$, and
    \item None of $\{u_{i_1},\cdots, u_{i_{r-1}}\}$ is selected in any of the previous steps.
\end{itemize}

At each step $i$, there are at most $r(i-1)n^{r-2}+sn^{r-2}$ edges in $L_{\mathcal{H}}(v_i)$ that contain a vertex already used in the earlier selections or in $V_0$. Since $d_{\mathcal{H}}(v_i) \geq r(s+1)n^{r-2}+1$,  such an edge $u_{i_1}\cdots u_{i_{r-1}}$ always exists. This gives a matching of size $s+1$, contradicting the assumption.  
\end{proof}

We also need the stability result for $\mathcal{K}_{\ell+1}^r$-free $r$-graph given by Mubayi \cite{M06}.
\begin{theorem}[Mubayi~\cite{M06}]\label{stability}
Fix $\ell \geq r \geq 2$ and $\delta >0$. Then there exists an $\epsilon >0$ and an $n_0$ such that the following holds for all $n > n_0$, if $G$ is an $n$-vertex $\mathcal{K}_{\ell+1}^r$-free $r$-graph with at least $t_r(n,\ell)-\epsilon n^r$ edges, then $G$ can be transformed to $T_r(n,\ell)$ by adding and deleting at most $\delta n^r$ edges.
\end{theorem}

%%%%%%%%%%%%%%%%%%%%%%%%%%%%%%%%%%%%
\section{Proof of Theorem~\ref{main theorem}}
In this section, we give a proof of Theorem~\ref{main theorem}. We always assume that $n$ is sufficiently large. Let $\mathcal{H}$ be an $n$-vertex $\mathcal{K}_{\ell+1}^r$-free $r$-graph with matching number $\nu(\mathcal{H}) \leq s$ and with the maximum number of edges. Since $\mathcal{G}$ is $\mathcal{K}_{\ell+1}^r$-free and $\nu(\mathcal{G})\leq s$. We have that $|\mathcal{H}| \geq |\mathcal{G}|$. First, we give some observations of $\mathcal{H}$.
\begin{lemma}\label{Link graph is Klr-free}
For any $v\in V(\mathcal{H})$, the link graph $L_{\mathcal{H}}(v)$ is $\mathcal{K}_{\ell}^{r-1}$-free.
\end{lemma}
\begin{proof}
    Assume by contradiction that $L_{\mathcal{H}}(v)$ contains a copy $S$ as a member of $\mathcal{K}_{\ell}^{r-1}$. Let $K$ be the core of $S$. Then by enlarging every edge of $S$ to contain $v$, we obtain a copy $S'\in \mathcal{K}_{\ell+1}^{r}$ and $K\cup\{v\}$ is the core of $S'$, a contradiction. 
\end{proof}

Define
$$
V_0 := \left\{v\in V(\mathcal{H})\colon d_{\mathcal{H}}(v) \geq  r(s+1)n^{r-2}+1 \right\}
$$
and let $U = V(\mathcal{H})\setminus V_0$. Next, we use Theorem~\ref{exact turan huaK} and Lemma~\ref{Hypergraph Vizing} to determine the size of $V_0$. 
\begin{lemma}\label{lem:upper bound V0}
We have $|V_0|=s$.
\end{lemma}
\begin{proof}
Suppose, on the contrary, that $|V_0| \leq s-1$. By Lemma \ref{Hypergraph Vizing}, we obtain that
\begin{align*}
|\mathcal{H}[U]|&\leq \chi'(\mathcal{H}[U])\cdot \nu(\mathcal{H}[U])\\ &\leq (r(\Delta(\mathcal{H}[U])-1) +1 )s  \\
&\leq r^2s(s+1)n^{r-2}+s.
\end{align*}
 Now we consider the $(r-1)$-graph $L_{\mathcal{H}}(v)$ for each $v \in V_0$. Recall that $\mathcal{H}$ is $\mathcal{K}_{\ell+1}^r$-free and $\nu (\mathcal{H})\leq s$. By Lemma~\ref{Link graph is Klr-free}, $L_{\mathcal{H}}(v)$ is $\mathcal{K}_\ell^{r-1}$-free. Then by Theorem \ref{exact turan huaK}, we get that for each $v \in V_0$, $$|L_{\mathcal{H}}(v)| \leq t_{r-1}(n-1, \ell-1) \leq \binom{\ell-1}{r-1} \left(\frac{n-1}{\ell-1}\right)^{r-1}.$$ Therefore, %we can calculate that
\begin{align*}
|\mathcal{H}| &\leq \sum_{v\in V_0}d_{\mathcal{H}}(v) + |\mathcal{H}[U]|  \\
&\leq (s-1)\binom{\ell-1}{r-1} \left(\frac{n-1}{\ell-1}\right)^{r-1} + r^2s(s+1)n^{r-2}+s \\
&<s\binom{\ell-1}{r-1} \left(\frac{n-s}{\ell-1}\right)^{r-1}=|\mathcal{G}|
\end{align*}
when $n$ is large enough. Hence, $|V_0|\geq s$ and by Lemma~\ref {Number of Max Degree is Bounded}, we have $|V_0|=s$.
\end{proof}
Next, we aim to obtain more structural details of $U$.
\begin{lemma}\label{U is Weakly independent set}
$U$ is a weakly independent set.
\end{lemma}
\begin{proof}
If there exists $e \in \mathcal{H}[U]$. Let $U'$ be the vertex set after deleting the vertices of $e$ from $U$. By a similar proof of Lemma \ref{Number of Max Degree is Bounded}, we can find $s$ matching in $\mathcal{H}[V_0 \cup U']$. Then this yields $(s+1)$ matching together with $e$, a contradiction.
\end{proof}
To obtain clearer structure of $U$, we state that the degree of every vertex in $V_0$ is quite large. For convenience, set $V_0 = \{v_1,\cdots,v_s\}$.
\begin{lemma}\label{lem:lower bound of v in Y}
There exists a constant $c=c(r,s,\ell)$ such that for each $v \in V_0$, $$d_U(v) \geq \binom{\ell-1}{r-1}\left(\frac{n-s}{\ell-1}\right)^{r-1}-cn^{r-2}.$$ 
\end{lemma}
\begin{proof}
First, we calculate the lower bound of $\sum_{v \in V_0}d_U(v)$. Clearly, for any $k\geq 2$ and $k$ vertices $v_{i_1},\cdots,v_{i_k}\in V_0$, we have 
$d_{U}(v_{i_1},\cdots,v_{i_k})\leq n^{r-k}.$ Combining 
\begin{align*}
|\mathcal{H}| &\leq \sum_{v_i \in V_0}d_U(v_i)+ \sum_{(v_i,v_j)\in \binom{V_0}{2}}d_U(v_i,v_j) + \cdots + \sum_{(v_{i_1},\cdots,v_{i_r}) \in \binom{V_0}{r}}1\\ &\leq\sum_{v_i\in V_0}d_{U}(v_i)+\sum_{k=2}^{r}\binom{s}{k}n^{r-s},
\end{align*}
and
$$
|\mathcal{H}| \geq |\mathcal{G}| = s\binom{\ell-1}{r-1}\left(\frac{n-s}{\ell-1}\right)^{r-1},
$$
we obtain that there exists a constant $c_1< 2^{s}$ such that 
\begin{align}\label{equ:lower of sum degree v}
\sum_{v_i \in V_0}d_U(v_i) \geq s\binom{\ell-1}{r-1}\left(\frac{n-s}{\ell-1}\right)^{r-1} - c_1n^{r-2}.
\end{align}
Choose $c\geq 2^{s}$. Suppose without loss of generality that $d_U(v_1) \leq \binom{\ell-1}{r-1}\left(\frac{n-s}{\ell-1}\right)^{r-1}-cn^{r-2}$.
Then
\begin{align}\label{equ:upper of sum degree v}
\sum_{v_i\in V_0}d_U(v_i) &\leq d_U(v_1) + \sum_{v_i \in V_0 \setminus v_1}d_U(v) \notag \\
& \leq \binom{\ell-1}{r-1}\left(\frac{n-s}{\ell-1}\right)^{r-1}-cn^{r-2} + (s-1)\binom{\ell-1}{r-1}\left(\frac{n-s}{\ell-1}\right)^{r-1}\notag  \\&\leq s\binom{\ell-1}{r-1}\left(\frac{n-s}{\ell-1}\right)^{r-1}-cn^{r-2}.
\end{align}
The inequalities \eqref{equ:lower of sum degree v} and \eqref{equ:upper of sum degree v} imply a contradiction by the choice of $c$ and $c_1$.
\end{proof}
%By the fact that $L_U(v)$ is $\mathcal{K}_\ell^{r-1}$-free for each $v \in V_0$, Theorem \ref{stability} and Lemma \ref{lem:lower bound of v in Y}, $L_U(v)$ can be transformed to $T_{r-1}(n-s,\ell-1)$ by adding and deleting at most $o(n^{r-1})$ edges.
Since each $L_U(v)$ is $\mathcal{K}_\ell^{r-1}$-free for $v \in V_0$, we may apply Theorem~\ref{stability} and Lemma~\ref{lem:lower bound of v in Y}, to conclude that $L_U(v)$ can be converted to $T_{r-1}(n-s,\ell-1)$ by adding or deleting at most $o(n^{r-1})$ edges. More precisely, for every $v_i \in V_0$, the set $U$ admits a partition $V_{1}^{i}\cup\cdots\cup V_{\ell-1}^{i}$ satisfying $|V_k^{i}| = \frac{1+o(1)}{\ell-1}n$ for $k \in [\ell-1]$, where the number of crossing $(r-1)$-hyperedges (those containing exactly one vertex from each part) is at least $t_{r-1}(n-s,\ell-1) - o(n^{r-1})$. Next, we prove that for any two vertices $v_i, v_j \in V_0$, their induced partitions of $U$ are nearly identical. Specifically, we can establish the following lemma.

\begin{lemma}\label{Partition are almost the same}
    For any two distinct vertices $v_i, v_j \in V_0$ and any integers $p, q \in [\ell-1]$, the intersection size satisfies either $|V_p^i \cap V_q^j| = o(n) $ or $|V_p^i \cap V_q^j| = \left(\frac{1+o(1)}{\ell-1}\right)(n-s).$
\end{lemma}
\begin{proof}
Assume for contradiction and without loss of generality that there exists $\alpha \in (0,1)$ such that for some indices $i,j$, the intersection satisfies $\left|V_1^{i} \cap V_1^{j}\right| = \left(\frac{\alpha + o(1)}{\ell-1}\right)(n-s)$. Define $V_{1,1}^{i} = V_{1}^{i} \cap V_{1}^{j}$ and $V_{1,2}^{i} = V_{1}^{i} \setminus V_{1,1}^{i}$ These satisfy $|V_{1,1}^{i}| = \frac{\alpha + o(1)}{\ell-1}(n-s)$ and consequently $|V_{1,2}^{i}| = \frac{1-\alpha + o(1)}{\ell-1}(n-s)$.
For each $k \in \{1,2\}$, the number of crossing $(r-1)$-hyperedges among $V_{1,k}^{i}, V_{2}^{i}, \ldots, V_{\ell-1}^{i}$ equals $$\left| T_{r-1}\left(V_{1,k}^{i}, V_{2}^{i}, \ldots, V_{\ell-1}^{i}\right) \right| + o(n^{r-1}).$$
Now select vertices uniformly and independently at random: $u_{1,1} \in V_{1,1}^{i}$, $u_{1,2} \in V_{1,2}^{i}$, $u_2 \in V_2^i$, $\ldots$, $u_{\ell-1} \in V_{\ell-1}^i$. Define three events as followings: event $A$ occurs when $\{u_{1,1},u_2,\ldots,u_{\ell-1}\}$ forms a $T_{r-1}(\ell-1,\ell-1)$ in $L_U(v_i)$; event $B$ occurs when $\{u_{1,2},u_2,\ldots,u_{\ell-1}\}$ forms such a hypergraph; and event $C$ occurs when there exist hyperedges in $\mathcal{H}$ containing $v_j,u_{1,1}$ and $u_{1,2}$.

\begin{claim}\label{Probability_Partitions_are_same_Claim}
We have (a) $\mathbb{P}(A) \geq 1 - o(1)$, (b) $\mathbb{P}(B) \geq 1 - o(1)$, and (c) $\mathbb{P}(C) \geq 1 - o(1)$.
\end{claim}

\begin{proof}[Proof of Claim~\ref{Probability_Partitions_are_same_Claim}]
Notice that we have already separated $U$ into $\ell$ parts, i.e. $V_{1,1}^i, V_{1,2}^i, V_2^i, \cdots, V_{\ell-1}^i$. For any fixed $r-1$ parts which do not include $V_{1,2}^{i}$ and any vertices $\{x_1,\ldots,x_{r-1}\}$ selected uniformly and independently from these parts, we consider two cases. If $V_{1,1}^{i}$ is chosen, then suppose, say $x_1 \in V_{1,1}^i$, we obtain that the probability\begin{align}\label{Same Inequality to prove c}
\mathbb{P}({x_1x_2\cdots x_{r-1} \in L_{U}(v_i) })\geq 1-\frac{o(n^{r-1})}{\alpha\left(\frac{n-s}{\ell-1}\right)^{r-1} + o(n^{r-1})}=1-o_{\ell,\alpha,r}(1).
\end{align}
Otherwise, we have that
$$\begin{aligned}
\mathbb{P}({x_1x_2\cdots x_{r-1} \in L_{U}(v_i) })\geq 1-\frac{o(n^{r-1})}{\left(\frac{n-s}{\ell-1}\right)^{r-1} + o(n^{r-1})}=1-o_{\ell,\alpha,r}(1).
\end{aligned}$$
Applying this to all $\binom{\ell-1}{r-1}$ possible $(r-1)$-tuples gives $$
\begin{aligned}
\mathbb{P}\left(A\right) 
&\geq 1 - \sum_{\substack{x_1,\cdots,x_{r-1} \\ \subseteq \{u_{1,1},u_2,\cdots,u_{\ell-1}\}}} \mathbb{P}\left(\{x_1x_2\cdots x_{r-1}\} \notin L_U(v_i)\right) \\
&\geq 1 - \binom{\ell-1}{r-1} \cdot o_{\ell,\alpha,r}(1) = 1 - o(1).
\end{aligned}
$$ The proof for (b) follows similarly by replacing $V_{1,1}^{i}$ with $V_{1,2}^{i}$, and $\alpha$ with $1-\alpha$.

For (c), let $V^k_{1,2} = V_{1,2} \cap V_k^j$ for $2\leq k \leq \ell-1$. Then we can calculate that 
\begin{align*}
\mathbb{P}(C) &= \sum^{\ell-1}_{k=2}\mathbb{P}(C|u_{1,2} \in V^k_{1,2}) \cdot \mathbb{P}(u_{1,2} \in V^k_{1,2})  \\
&\geq\sum^{\ell-1}_{k=2} \left(1- \frac{o(n^{r-1})}{\binom{\ell-1}{r-3}|V_{1,1}||V^k_{1,2}|\left(\frac{n-s}{\ell-1}\right)^{r-2} + o(n^{r-1})} \right) \cdot \frac{|V^k_{1,2}|}{|V_{1,2}|} \\
&= 1-o(1).
\end{align*}
\end{proof}
 By Claim~\ref{Probability_Partitions_are_same_Claim}, there exists a selection $u_{1,1}, u_{1,2}, u_2, \cdots, u_{\ell-1}$ where events $A$, $B$, and $C$ all occur. Let $e \in E(\mathcal{H})$ be a hyperedge containing $\{v_j, u_{1,1}, u_{1,2}\}$ (guaranteed by $C$). Constructing $F$ by adding $e$ to $\mathcal{H}[K]$ where $K = \{v_i, u_{1,1}, u_{1,2}, u_2, \ldots, u_{\ell-1}\}$ yields a contradiction that $F$ is contained in $\mathcal{K}_{\ell+1}^r$ with $K$ as its core, violating our initial assumptions. 
\end{proof}
By Lemma~\ref{Partition are almost the same}, every vertex $v \in V_0$ induces an almost identical partition of $U$. Without loss of generality, we assume that for any $v_i, v_j \in V_0$ and $p, q \in [\ell-1]$, if $p = q$ then $|V_{p}^{i} \cap V_{q}^{j}| = \frac{1 + o(1)}{\ell-1}(n - s)$, otherwise $|V_{p}^{i} \cap V_{q}^{j}| = o(n)$. We now show that $V_0$ forms a strongly independent set in $\mathcal{H}$. 
\begin{lemma}\label{V0 is strongly independent set}
   For any $e\in \mathcal{H}$, $|e\cap V_0|\leq 1$.
\end{lemma}
\begin{proof}
 Assume by contradiction that $v_i,v_j\in V_0$ and there is an edge $e_0\in \mathcal{H}$ such that $v_i,v_j\in e_0$. 
 For each $p \in [\ell-1]$, let $V'_p = \bigcap_{i=1}^{s} V_{p}^{i}\setminus e_0$ and by Lemma~\ref{Partition are almost the same}, we have that $|V'_{p}|=\frac{1+o(1)}{\ell-1}(n-s)$. The number of crossing hyperedges among $\{V'_p:p\in [\ell-1]\}$ is $$|T_{r-1}(V'_1,\cdots,V'_{\ell-1})|+o(n^{r-1})=\binom{\ell-1}{r-1}\left(\frac{n-s}{\ell-1}\right)^{r-1}+o(n^{r-1}).$$   
 We select $\ell-1$ vertices $v'_p\in V'_p$ for $p\in [\ell-1]$ uniformly and independently at random. With a similar probabilistic argument, for $k\in \{i,j\}$ we have  $$\mathbb{P}(\{v'_1,\cdots,v'_{\ell-1} \text{ form a }T_{r-1}(\ell-1,\ell-1) \text{ in }L_{\mathcal{H}}(v_k)\})\geq 1-o(1).$$
Consequently, we can select vertices $\{v'_1,\dots,v'_{\ell-1}\}$ where any $r-1$ of them form a hyperedge with either $x_i$ or $x_j$ in $\mathcal{H}$. Define $K = \{v_i, v_j, v'_1,\dots,v'_{\ell-1}\}$ and construct $S$ as the subhypergraph formed by $\mathcal{H}[K]$ plus the edge $e_0$, we obtain that $S \in \mathcal{K}_{\ell+1}^r$ with $K$ as its core, which contradicts to our assumption that $\mathcal{H}$ is $\mathcal{K}_{\ell+1}^{r}$-free.
\end{proof}
Now we are ready to prove Theorem~\ref{main theorem}.
\begin{proof}[Proof of Theorem~\ref{main theorem}]
Let $\mathcal{H}$ be a $\mathcal{K}_{\ell+1}^{r}$-free graph with $n$ vertices and with $\nu(\mathcal{H})\leq s$. Furthermore, we assume that $|\mathcal{H}|\geq|\mathcal{G}(n,\ell,s,r)|$. By Lemma~\ref{Number of Max Degree is Bounded} and Claim~\ref{U is Weakly independent set}, we can partition $V(\mathcal{H})$ into $V_0$ and $U$ such that $|V_0|=s$ and $U$ is a weakly independent set. Moreover, by Lemma~\ref{V0 is strongly independent set}, every edge in $\mathcal{H}$ has at most one vertex in $V_0$. Thus, for each vertex $x\in V_0$, the link graph $L_{\mathcal{H}}(x)$ contains no vertex in $V_0$ and $L_{\mathcal{H}}(x)$ is $\mathcal{K_{\ell}}^{r-1}$-free. Notice that $|V(L_{\mathcal{H}}(x))|\leq n-s$ and by Theorem~\ref{exact turan huaK}, we have 
$$d_{\mathcal{H}}(x)=|L_{\mathcal{H}}(x)|\leq t_{r-1}(n-s,\ell-1).$$
Therefore, $$|\mathcal{H}|=\sum_{x\in V_{0}}d_{\mathcal{H}}(x)\leq \sum_{x\in V_0}t_r(n-s,\ell-1)=s\cdot t_{r-1}(n-s,\ell-1),$$
with equality holds if for each $x\in V_0$, $L_{\mathcal{H}}(x)=T_{r-1}(n-s,\ell-1)$ and hence, the vertex $x$ admits a partition $U=V_1^{x}\cup\cdots\cup V_{\ell-1}^{x}$. If there exist two vertices $x,y \in V_0$ admitting different partitions of $U$, we may assume without loss of generality that $v_{1,1}, v_{1,2} \in V_1^x$ lie in different classes under $y$'s partition, implying the existence of a hyperedge $e \in E(\mathcal{H})$ containing $\{y, v_{1,1}, v_{1,2}\}$. Now we can select $v_i \in V_i^x$ for $2 \leq i \leq \ell-1$ (possibly in $e$) and define that $K = \{x, v_{1,1}, v_{1,2}, v_2, \dots, v_{\ell-1}\}$ with $S = \mathcal{H}[K] \cup \{e\}$. Then we obtain that $S \in \mathcal{K}_{\ell+1}^r$ where $K$ forms a core, yielding a contradiction. Thus all $v \in V_0$ induce identical partitions, proving that $\mathcal{H} = \mathcal{G}(n,\ell,s,r)$.
\end{proof}

%%%%%%%%%%%%%%%%%%%%%%%%%%%%%%%%%%%%%
\section{Further results and discussions}
We first provide a proof of Theorem~\ref{turan fano plane}. 

\begin{proof}[Proof of Theorem~\ref{turan fano plane}]
Let $\mathcal{G}$ be an edge maximum $\mathbb{F}$-free $3$-uniform graph with matching number at most $s$. Assume further that $|\mathcal{G}| \ge |\mathcal{F}(n,s)|$. Let $X \subseteq V(\mathcal{G})$ be the set of vertices in $\mathcal{G}$ whose degrees are at least $3(s+1)n + 1$, and let $Y = V(\mathcal{G}) - X$. Since $\Delta(\mathcal{G}[Y])\leq 3(s+1)n$, by Lemma~\ref{Hypergraph Vizing} we have 
\begin{align*}
  |\mathcal{G}[Y]|&\leq s \cdot \left(3 \cdot (\Delta(\mathcal{G}[Y]-1)+1) \right) \\
  &\leq 9s(s+1)n-2s.
\end{align*}
We first claim that $|X| = s$. By Lemma~\ref{Number of Max Degree is Bounded}, we have $|X| \leq s$. Furthermore, if $|X| \le s-1$, then since $n\geq 20s(s+1)$, we have $$\begin{aligned}|\mathcal{G}|&\leq  \sum_{x\in X}d_{\mathcal{G}}(x)+|\mathcal{G}[Y]|\\&\leq (s-1)\binom{n-1}{2}+9s(s+1)n-2s\\&<s\binom{n-s}{2}+\binom{s}{2}(n-s)=|\mathcal{F}(n,s)|,\end{aligned}$$ which contradicts our assumption. Thus, we may assume that $X = \{x_1, x_2, \ldots, x_s\}$. 

We now show that $Y$ is a weakly independent set in $\mathcal{G}$. Otherwise, suppose that there exists an edge $e_0$ in $\mathcal{G}[Y]$. Since $d_{\mathcal{G}}(x_1) \ge 3(s+1)n +1 \ge \binom{s+2}{2}$, there exist vertices $u_1, v_1 \in Y \setminus e_0$ such that $\{x_1, u_1, v_1\}$ is an edge in $\mathcal{G}$.  Furthermore, suppose we have already found $i+1$ pairwise disjoint edges $E_i = \{e_j = \{x_j, u_j, v_j\} : 1 \le j \le i\} \cup \{e_0\},$ where $i \le s-1$. Consider $x_{i+1} \in X$. Since $d_{\mathcal{G}}(x_{i+1}) \ge 3(s+1)n + 1 > \binom{3s}{2} \ge \binom{s + 2(i+1)}{2}$, there exist vertices $u_{i+1}, v_{i+1}$ such that $\{x_{i+1}, u_{i+1}, v_{i+1}\}$ is an edge disjoint from all edges in $E_i$. Consequently, we can find in $\mathcal{G}$ a matching of size $s+1$, contradicting the assumption that $\nu(\mathcal{G}) \le s$. When $s \le 2$, the proposition holds straightforwardly. We now consider the case $s \ge 3$. It suffices to show that $X$ is a weakly independent set in $\mathcal{G}$. Otherwise, suppose that $e = \{x_1, x_2, x_3\}$ is an edge in $\mathcal{G}$. We define
$$
E_1' = \left\{ \{y_1, y_2\} \subset Y : y_1, y_2, x_i \text{ forms a hyperedge for } 1 \leq i \leq 3 \right\},
$$
and let $G_1$ be the simple graph induced by $E_1'$. Let $Y_1 = V(G_1)$, and let $Y_2 = Y \setminus Y_1$. Since $\mathcal{G}$ is $\mathbb{F}$-free, it follows that $G_1$ is $K_4$-free. Let $|Y_1| = t$. Then $
e(G_1) \leq \frac{1}{3} t^2,$ and thus the number of missing crossing hyperedges is at least 
$$
f(t)=\binom{t}{2} - \frac{1}{3} t^2 + \binom{|Y_2|}{2} + |Y_1|\cdot |Y_2|=\binom{t}{2} - \frac{1}{3} t^2 + \binom{n-s-t}{2}+(n-s-t)t.
$$ Since $f'(t)=-\frac{2}{3}t\leq 0$ for $t\in [0,n-s]$, we have $$f(t)\geq f(n-s)=\frac{1}{6} (n-s)(n-s-3).$$
%Partition the vertices of $Y$ arbitrarily into groups of four, denoted by $Y_1, Y_2, \ldots, Y_t$, where $t = \lfloor (n - s)/4 \rfloor$. Since $\mathcal{G}$ is $\mathbb{F}$-free, each $Y_i$ must lose at least one crossing edge with respect to $e$. 
Therefore, since $n\geq 20s(s+1)$, $$ |\mathcal{G}|\leq \binom{s}{3}+s\binom{n-s}{2}+\binom{s}{2}(n-s)- \frac{1}{6} (n-s)(n-s-3)< |\mathcal{F}(n,s)|,$$ which contradicts our assumption for $\mathcal{G}$. Thus, $\mathcal{F}(n,s)$ is the unique extremal hypergraph.
\end{proof}

We note that both $\mathrm{ex}_r(n, \{\mathcal{K}_{\ell+1}^r, M_{s+1}^r\})$ and $\mathrm{ex}_3(n, \{\mathbb{F}, M_{s+1}^3\})$ can be viewed as extensions of the problems for $\mathcal{K}_{\ell+1}^r$ (see~\cite{M06}) and $\mathbb{F}$ (see~\cite{BR19}), respectively. In each case, one can regard the extremal construction as obtained by contracting one part (whose size depends on $n$) of $\mathcal{K}_{\ell+1}^r$ or $\mathbb{F}$ into $s$ vertices, so that every hyperedge in the extremal hypergraph intersects this part, thereby ensuring that the matching number does not exceed $s$. 

However, this is not always the case. For instance, consider the hypergraph $H_{\ell+1}^r$ with $\ell \ge r \ge 3$. By the work of Pikhurko~\cite{pikhurko2013Hl}, we know that for sufficiently large $n$, the extremal configuration of $\mathrm{ex}_r(n, H_{\ell+1}^r)$ coincides with that of $\mathrm{ex}_r(n, \mathcal{K}_{\ell+1}^r)$. Now we further consider the case where the matching number is at most $s$. When $\ell - 1 \le s \le \binom{\ell}{2}-1$, we construct a hypergraph $\mathcal{H}$ consisting of all edges defined in Definition~\ref{def of Gn,l,s,r} together with one additional $r$-edge that contains exactly $r-1$ vertices from $V_0$. It is easy to verify that this hypergraph is still $H_{\ell+1}^r$-free, and since every edge intersects $V_0$, its matching number remains $s$. However, the number of edges in this hypergraph equals $s \cdot t_{r-1}(n-s, \ell-1) + 1$, which shows that the extremal hypergraph is not $G(n,\ell,s,r)$.

Therefore, we propose the following conjecture.
\begin{conj}\label{conj for H}
Let $\ell \ge r \ge 3$ and $s \ge \binom{\ell}{2}$ be integers. 
For sufficiently large $n$, we have
$$
\mathrm{ex}_r(n, \{H_{\ell+1}^r, M_{s+1}^r\}) = s \cdot t_{r-1}(n-s, \ell-1).
$$
In particular, $\mathcal{G}(n,\ell,s,r)$ is the unique extremal hypergraph.
\end{conj}

%%%%%%%%%%%%%%%%%%%%%%%%%%%%%%%%%%%%%
%\section{Concluding remarks}\label{SEC:Remarks}
 
%%%%%%%%%%%%%%%%%%%%%%%%%%%%%%%%%%%%%%%%%%%%%%%%
%%%%%%%%%%%%%%%%%%%%%%%%%%%%%%%%%%%%%%%%%%%%%%%%%%
\section*{Acknowledgements}
This project was initiated during the visit of the second author to the Institute for Basic Science (IBS). We gratefully acknowledge the Extremal Combinatorics and Probability Group at IBS for organizing the 3rd ECOPRO Student Research Program. In particular, Caihong Yang and Jiasheng Zeng would like to express their sincere appreciation to Professor Hong Liu for providing this exceptional research opportunity and for many fruitful discussions that greatly inspired and shaped our work. Caihong Yang would like to thanks Professor Jianfeng Hou for his helpful
comments and interesting directions of this problem.  Caihong Yang is supported by the Institute for Basic Science (IBS-R029-C4), the National Key R\&D Program of China (Grant No.~2023YFA1010202), and the Central Guidance on Local Science and Technology Development Fund of Fujian Province (Grant No.~2023L3003), Xiao-Dong Zhang is supported by  the National Natural Science Foundation of China (No.12371354) and the Montenegrin-Chinese Science and Technology (No.4-3).

\bibliographystyle{abbrv}
\bibliography{matching}
 
\end{document}